\documentclass[12pt]{amsart}
\usepackage{amsmath,amsthm,amsfonts,amssymb,times,latexsym,mathabx,url}
\usepackage[pdftex]{graphicx}
\usepackage{ bbold }

\graphicspath{}
\DeclareGraphicsExtensions{.pdf,.png,.jpg}
\newtheorem{theorem}{Theorem}[section]

\newtheorem{lem}[theorem]{Lemma}

\numberwithin{equation}{section}

\renewcommand{\leq}{\leqslant}
\renewcommand{\geq}{\geqslant}

\newcommand{\up}{\uparrow}

\makeatletter
\renewcommand{\pmod}[1]{\allowbreak\mkern7mu({\operator@font mod}\,\,#1)}
\makeatother

\renewcommand{\leq}{\leqslant}
\renewcommand{\geq}{\geqslant}

\renewcommand{\leq}{\leqslant}
\renewcommand{\geq}{\geqslant}

\numberwithin{equation}{section}

\begin{document} 
	
	\title{On The Tree Structure of Natural Numbers, II}
	\author{ Vitalii V. Iudelevich}
	\address{Moscow State University, Leninskie Gory str., 1, Moscow, Russia, 119991}
	\email{vitaliiyudelevich@mail.ru}

	\begin{abstract}
		Each natural number can be associated with some tree graph. Namely, 
		a natural number $n$ can be factorized as 
		$$ n = p_1^{\alpha_1}\ldots p_k^{\alpha_k},$$
		where
		$p_i$ are distinct prime numbers. Since
		$\alpha_i$ are naturals, they can be factorized in such a manner as well.
		This process may be continued, building the ''factorization tree'' until all
		the top numbers are $1$. Let $H(n)$ be the height of the tree corresponding to the number $n$, and let the symbol $\up\up$ denote tetration. In this paper, we derive the asymptotic
		 formulas for the sums
		$$\mathcal{M}(x) = \sum_{p\leqslant x} H(p-1),\ \ \mathcal{H}(x) = \sum_{n\leqslant x}2\up\up H(n),$$
		and
		$$\mathcal{L}(x) = \sum_{n\leqslant x}\dfrac{2\up\up H(n)}{2\up\up H(n+1)},$$
		where the summation in the first sum is taken over primes.
	\end{abstract}
	
	\date{\today}
	\maketitle
	
	\section{Introduction}
	Every natural number can be uniquely factorized into a product of prime powers by the Fundamental Theorem of Arithmetic. The prime powers themselves are integers and thus can be factorized too. We can continue this procedure until all the prime powers become prime. We  call the obtained decomposition the \emph{super-factorization}\footnote{The definition was proposed in 2003 by L.~Quet.} of the initial natural number.
	 The following expression is the super-factorization of the number $N = 2^{50}3^5,$
	$$N = 2^{2\cdot5^2}\cdot3^{5}.$$
	The super-factorization of the number $n$ can be visualized using a tree graph.
	The construction of this tree, denoted $T_n$, starts with the root labeled by the number $n$. The root gives a rise to as many edges as $n$ has distinct prime divisors. Each edge is labeled by the corresponding prime divisor and connects the root with a new vertex labeled by the power of the prime in the decomposition of $n$. This operation is repeated for each of the newly formed vertices. The construction of the tree ends when all the vertices that have arisen at the next step are labeled by $1$. Let $H(n)$ be the height of the tree $T_n$, that is, the length of the longest path leading from the root to a leaf vertex. As an example, Figure \ref{fig:image1} presents
	the tree associated with the number $N = 2^{50}3^5$.
	Figure \ref{fig:image2} illustrates the unlabeled trees associated with integers from $1$ to $34$.
	\begin{figure}[h]
		\caption{The graph $T_{N}$ with $ N = 2^{50}3^5.$ }
		\centering{\includegraphics[scale=0.9]{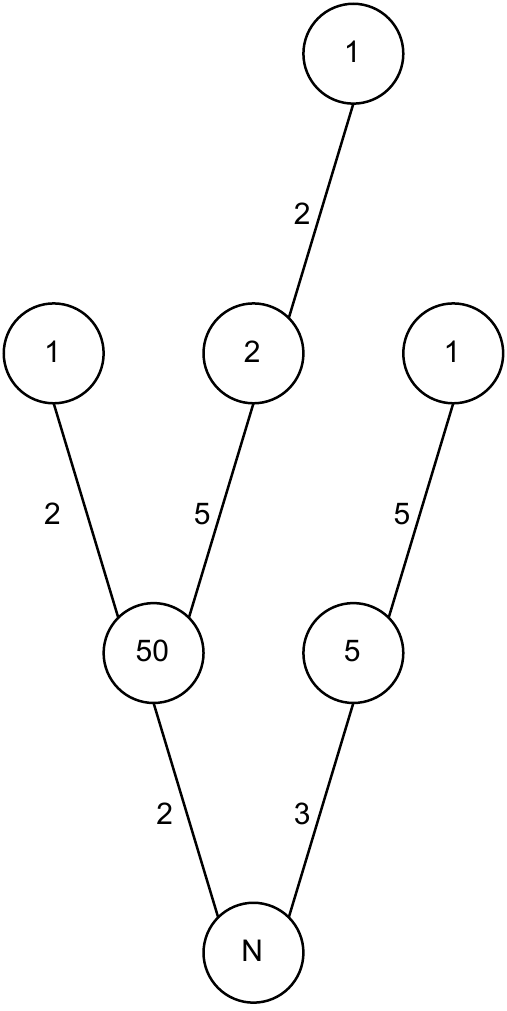}}
		\label{fig:image1}
	\end{figure}
		\begin{figure}[h]
		\caption{The unlabeled graphs $T_{n}$ with $1\leqslant n \leqslant 34$. }
		\centering{\includegraphics[scale=1.2]{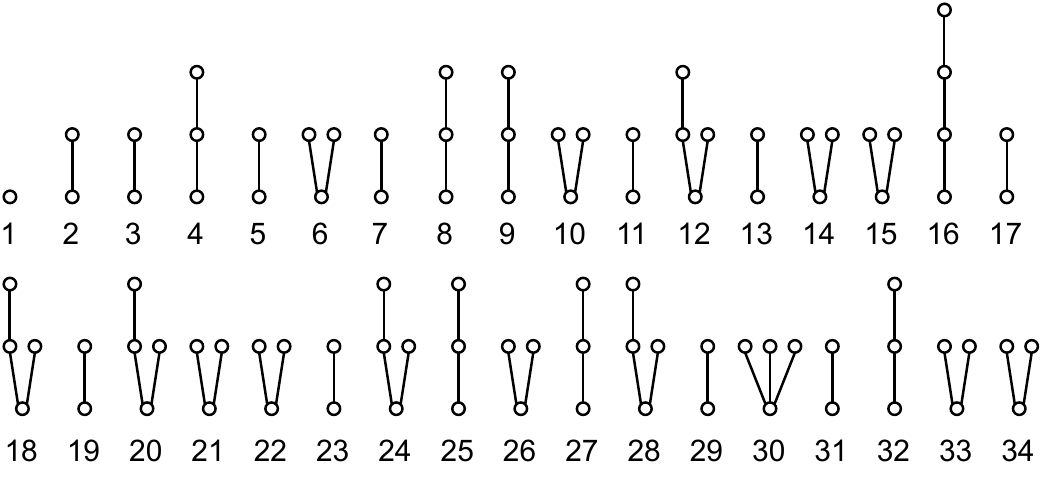}}
		\label{fig:image2}
	\end{figure}
 
The tetration, denoted $\up\up$, is defined as follows,
	$$a\up\up b = {a^{a^{\ldots^a}}},$$
	where $a$ and $b$~ are natural numbers, and the length of the power tower is $b$. For convenience, for $a\geqslant 1$ we set $a\up\up 0 = 1.$ Let $\log^*$ be the base-$2$ iterated logarithm, that is, for $x\geqslant 2$ and integer $m\geqslant 1$, the equality $\log^* x = m$ is equivalent to
	$$2\up\up m\leqslant x< 2\up\up(m+1).$$
	The unpublished book \cite{Trap-Rob} gives two problems concerning the function $H(n)$. These problems were suggested by E.~Snowden to the Mathematical Club of the University of Maryland. We formulate them below:
	\begin{enumerate}
		\item For $n\geqslant 1$ we define $$D_n = \lim_{N\to+\infty}\dfrac{\#\left\{k\leqslant N: H(k)\geqslant n\right\}}{ N}.$$
		Then the inequality
		\begin{equation}\label{ineqdn}
			\dfrac{1}{2}<(2\up\up n)D_n\leqslant 3
		\end{equation}
	holds for all $n\geqslant 1.$
		\item Let 
		\begin{equation}\label{H}
			\mathcal{F} = \lim_{N\to +\infty}\dfrac{1}{N}\sum_{n\leqslant N} H(n),
		\end{equation}
		then we have
		$$1.42333<\mathcal{F}<1.4618.$$
	\end{enumerate}
	In \cite{Iud21}, we obtained explicit formulas for $D_n$, computed the value $\mathcal{F} = 1.436157...$, and proved \eqref{H} by providing a stronger statement,
	\begin{equation}\label{HHH}\sum_{n\leqslant x} H(n) = \mathcal{F}x+O\left(x^{1/2}(\log x)^{5/4+\varepsilon}\right),\ \ \varepsilon>0.
		\end{equation}
In this paper, we find the asymptotics for the sum
$$ \mathcal{M}(x) = \sum_{p\leqslant x} H(p-1),$$
where $p$ runs through primes. We prove the following theorem.
As usual, $\pi(x)$ denotes the number of primes up to $x$.
\begin{theorem}\label{T1}
	We have
	\[
	\mathcal{M}(x) = c_0 \pi(x) + O_A \left(\dfrac{x}{(\log x)^A} \right)
	\]
	for any fixed $A \geqslant 2$, where $c_0 = 1.7083\ldots$
\end{theorem}
In \cite{Iud21}, we used the complex integration method to provide the estimates related to $H(n)$. However, this method no longer works for the case of $\mathcal{M}(x)$. Instead, one of the key ingredients for establishing Theorem \ref{T1} will be the identity
\begin{equation}\label{MainIdent}
	\mathbb{1}_{H(n)\leqslant k} = \sum_{d|n}\mathcal{F}_k(d),
\end{equation}
where $\mathcal{F}_k(n)$ are some functions having \emph{small support} to be defined later. This equality is a generalization of the well-known identity $\mu^2(n) = \sum_{d^2|n}\mu(d),$ where $\mu(n)$ denotes the M$\ddot{\text{o}}$bius function. The same equality allows us to find an elementary proof (without using complex analysis) of \eqref{HHH}.

Another result of this paper arises from attempts to generalize the main result of \cite{Iud21}. Our methods can be extended to the sums of the form $\sum_{n \leqslant x} f(H(n))$, provided that the function $f(n)$ does not grow too fast. In the case of a rapidly growing function $f(n)$ (say, $f(n) = 2\up\up n$), one has to develop a different approach. We prove the following theorem.
\begin{theorem}\label{T2}
	Set $\mathcal{H}(x) = \sum_{n \leqslant x} 2 \uparrow \uparrow H(n)$, then
	\begin{equation}\label{HH}
		\mathcal{H}(x) = \dfrac{x \log^* x}{2} + c_1x + H_1(x) + O \left( x^{1/2} \left( \log x\right) ^{{9}/{4}} \right), 
	\end{equation}
	where \[
	c_1 = \sum\limits_{k=1}^{+\infty} \left((2 \uparrow \uparrow k) \mathcal{D}_k - \dfrac{1}{2} \right)  = 1.813 \ldots,
	\]
	\[
	H_1(x) = 2 \uparrow \uparrow \log^* x \left( \mathcal{D}_{\log^* x}(x) - x\mathcal{D}_{\log^* x}\right), 
	\]
	and $\mathcal{D}_k(x)$, $\mathcal{D}_k$ are defined in \eqref{DDkx} and \eqref{Dkx} respectively.
	Moreover, the following inequalities hold:
	\begin{itemize}
		\item $H_1(x) \ll x$; 
		\item $H_1(x) \ll x^{1/2}\left( \log x\right) ^{{9}/{4}}$, if $x = (2 \uparrow \uparrow N)-1$ for some natural $N$;
		\item $H_1(x) = {x}/{2} + O(x^\Delta)$, where $\Delta = 2 - \frac{\log 3}{\log 2} = 0.41503\ldots$, if $x = 2 \uparrow \uparrow N$ for some natural $N$. 
	\end{itemize}
\end{theorem}
Equality \eqref{MainIdent} also allows us (see Lemma \ref{L8}) to find the asymptotics for
\begin{equation}\label{HHkmx}
	H_{k,m}(x) = \#\left\{n\leqslant x: H(n)\leqslant k, H(n+1)\leqslant m\right\},
\end{equation}
coinciding for $k=m=1$ with the sum
$$\mathcal{E}(x) = \sum_{n\leqslant x}\mu^2(n)\mu^2(n+1).$$
The sum $\mathcal{E}(x)$ was previously studied in \cite{Carl32}, \cite{Heath-Br84} and \cite{Reus14}. Trying to keep the paper elementary, we do not use the deep methods of these papers when studying $H_{k,m}(x).$ For the same reason, some lemmas in our paper are not proved in the strongest form. From the asymptotic formula for $H_{k,m}(x)$ we deduce the following theorem.
\begin{theorem}\label{T3} Set
	\[
	\mathcal{L}(x) = \sum\limits_{n \leqslant x} \dfrac{2 \uparrow \uparrow H(n)}{2 \uparrow \uparrow H(n+1)},
	\]
	then we have
	\begin{equation}\label{LL(x)}
		\mathcal{L}(x) = c_2 x \log^*x + c_3x + L_1(x) + O \left( x^{{5}/{8}} (\log x)^3 \right),
	\end{equation}
	where $c_2$ is given by\footnote{The number $c_2$ is the sum of a rapidly convergent series of rational numbers, and hence is a Liouville number; so $c_2$ is transcendental.} 
	$$c_2= \dfrac{457}{1944}+ \dfrac{1}{3} \sum\limits_{m=3}^{+\infty} \left( \dfrac{1}{2 \uparrow \uparrow m} \left(\dfrac{1}{3^{2 \uparrow \uparrow (m-1)}} - \dfrac{1}{3^{2 \uparrow \uparrow m}}\right) \right) = 0.23533\ldots,$$
	$c_3$ is some constant, the function $L_1(x)$ is defined as
	\begin{equation}L_1(x) = 2 \uparrow \uparrow \log^* x \sum\limits_{m=1}^{+\infty} \dfrac{\mathcal{D}_{\log^* x,m}(x) - x \mathcal{D}_{\log^* x, m}}{2 \uparrow \uparrow m},
	\end{equation} 
	 and $\mathcal{D}_{k,m}(x)$, $\mathcal{D}_{k,m}$ are defined in \eqref{DDkmx} and \eqref{Dkm}, respectively.
	Moreover, for $L_1(x)$ the following inequalities hold:
	\begin{itemize}
		\item $L_1(x)\ll x$; 
		\item $L_1(x) \ll x^{5/8} (\log x)^3$, if $x = (2\up\up N) -1$ for some natural $N$. 
	\end{itemize}
\end{theorem}
We note that the methods used in our work can also be applied to the estimates of other sums. For example, proceeding in the same way as in the proof of Theorem \ref{T1}, one can prove that
$$\sum_{n \leqslant x}H(n^2+1) = cx(1+o(1)),$$
for some positive constant $c$.

\bigskip 

\textbf{Acknowledgements}. The work was supported by the
Theoretical Physics and Mathematics Advancement Foundation “BASIS”.

\section{Auxiliary Results}
We set $f_m(n) = \mathbb{1}_{H(n) \leqslant m}$, where $n \geqslant 1$ and $m \geqslant 0$. It is clear that $f_m(n)$ is a multiplicative function. For convenience, we set
$$f_0(\alpha) = \begin{cases}
	1, \ \text{if}\ \alpha\in\{0, 1\}, \\
	0, \ \text{otherwise}.
\end{cases}
$$
Then for any $m \geqslant 1$, $\alpha \geqslant 0$, and any prime $p$ we get
\begin{equation}\label{f_{m-1}}
	f_m(p^\alpha) = f_{m-1}(\alpha).
\end{equation}
We define the function $\mathcal{F}_m(n)$ by the equality
\begin{equation}\label{F_m(d)}
	f_m(n) = \sum\limits_{d \mid n} \mathcal{F}_m(d),
\end{equation}
and prove the following lemma for it.

\begin{lem}\label{L1}
	For $m \geqslant 1$ we have
	\begin{equation}\label{F_m(n)}
		\mathcal{F}_m(n) = \prod\limits_{p^\alpha || n} \left( f_{m-1}(\alpha) - f_{m-1}(\alpha - 1) \right),
	\end{equation}
	where $|\mathcal{F}_m(n)| \leqslant 1$, and if $\mathcal{F}_m(n) \ne 0$, then every prime dividing $n$ is occurs in the canonical decomposition of $n$ with a degree at least $2 \uparrow \uparrow m$.
	\end{lem}
\begin{proof}
	Using the M$\ddot{\text{o}}$bius inversion formula and equality \eqref{f_{m-1}}, we obtain equality \eqref{F_m(n)} from \eqref{F_m(d)}. Since $f_m(n) \in \left\{ 0, 1\right\}$ for $m, n \geqslant 0$, it follows that \eqref{F_m(n)} implies $|\mathcal{F}_m (n)| \leqslant 1$. Let $p^\alpha || n$, where $\alpha < 2 \uparrow \uparrow m$. Since $r = 2 \uparrow \uparrow m$ is the smallest number for which $H(r) = m$, then we have $f_{m-1}(\alpha) = f_{m-1 }(\alpha-1) = 1$. Hence equality \eqref{F_m(n)} implies $\mathcal{F}_m(n) = 0$. The claim follows.
\end{proof}

\begin{lem}\label{L2}
	Set $D_k(x) = \# \left\{ n \leqslant x: H(n) \geqslant k \right\}$, then for $x \geqslant 1$ and $k \geqslant 1$, we have
	\[
	D_k(x) = D_kx + O \left( x^{1/2}\left( \log x\right)^{{5}/{4}} \right),
	\]
	where the implied constant is absolute, $D_1 = 1$, and $D_k$ for $k \geqslant 2$ is defined by
	\[
	D_k = 1- \prod\limits_{p}\left( 1 + \sum\limits_{\alpha = 1}^{+\infty} \dfrac{f_{k-1}(p^\alpha) - f_{ k-1}(p^{\alpha-1})}{p^\alpha} \right).
	\]
	Moreover, for $x \geqslant x_0$ and $k \geqslant 1$, we have $D_k(x) \leqslant \dfrac{4x}{2 \uparrow \uparrow k}$ and $D_k \leqslant \dfrac{4} {2 \uparrow \uparrow k}$.
\end{lem}
\begin{proof}
	See Theorem 2 in \cite{Iud21}.
\end{proof}
Using Lemma \ref{L1} and \eqref{f_{m-1}} for $k \geqslant 2$ we obtain \[
D_k = 1 - \sum\limits_{d=1}^{+\infty} \dfrac{\mathcal{F}_{k-1}(d)}{d} = -\sum\limits_{d > 1 } \dfrac{\mathcal{F}_{k-1}(d)}{d}.
\]
Further, for $k \geqslant 1$ we set $\mathcal{D}_k = D_k - D_{k+1}$.
We also put
\begin{equation}\label{DDkx}
	\mathcal{D}_k(x) = \# \left\{ n \leqslant x: H(n) = k \right\}.
\end{equation}
Hence for $x \geqslant 1$ we have \begin{equation}\label{Dkx}
	\mathcal{D}_k(x) = \mathcal{D}_kx + O \left(x^{1/2}\left( \log x \right)^{{5}/{4}}\right),
\end{equation}
where the constant in the symbol $O$ is absolute.
\begin{lem}\label{L3}
	Let $m \geqslant 3$, $\alpha = 2 \uparrow \uparrow (m-1)$, and $n < 3^\alpha$. Then $H(n) = m$ if and only if $n$ has the form $n = 2^\alpha a$, where $a < 1.5^\alpha$ is odd. Moreover, if $n \ne 2^\alpha a$, then $H(n) \leqslant m-1$.
\end{lem}
\begin{proof}
	Let $H(n) = m$, then $n$ has the form
	\begin{equation}
		n = a_1 p_1^{{a_2p_2}^{{\ldots}^{a_mp_m}}},
	\end{equation}
	where $p_1, p_2, \ldots, p_m$ are primes, $H(a_1) \leqslant m$, $H(a_2) \leqslant m-1$, $\ldots$, $H(a_m) \leqslant 1$ and $p_1 \nmid a_1$, $p_2 \nmid a_2$, $\ldots$, $p_m \nmid a_m$.
	If $p_1 \geqslant 3$, then replacing $p_2, \ldots, p_m$ with $2$ and $a_1, \ldots, a_m$ with $1$, we get
	\begin{equation}\label{contr}
		n \geqslant 3^\alpha,
	\end{equation}
	which is false due to the assumption, hence $p_1=2$.
	
	Suppose that the set $\left\{p_2, \ldots, p_m\right\} $ contains some $p_k \geqslant 3$. Then we replace $p_k$ with $3$, and the remaining primes with $2$. If $k \leqslant m-1$, then successively applying the inequality
	\begin{equation}\label{ineq1}
		2^{3^a} \geqslant 3^{2^a}, \ a \geqslant 2,
	\end{equation}
	we obtain the contradictory inequality \eqref{contr}. If $k=m$, then applying the inequality $2^{2^3} \geqslant 3^{2^2}$ and then successively inequality \eqref{ineq1}, we again get \eqref{contr}. So, $p_1 = p_2 = \ldots = p_m = 2$.
	
	Suppose that the set $\left\{a_2, \ldots, a_m \right\}$ contains some $a_k \geqslant 2$. Then we replace $a_k$ with $2$, and the remaining $a_i$ with $1$. Then applying the inequality $2^{2a} \geqslant 3^a$, $a \geqslant 1$, and then inequality \eqref{ineq1} a finite number of steps , we get \eqref{contr}. So, $a_2 = \ldots = a_m=1$, and $n$ has the form
	\begin{equation}\label{represent}
		n = a_1 2^\alpha,
	\end{equation}
	where $a_1$ is odd. Since $n < 3^\alpha < 2^{2^\alpha}$, then $H(n) \leqslant m$. So, all numbers of the form \eqref{represent} and only these numbers have the height $m$. Hence, if $n \ne 2^\alpha a$, where $a$ is odd, then $H(n) \leqslant m-1$, as required.
\end{proof}
	\begin{lem}\label{L4}
		Let $Y \geqslant 20$ and $\alpha \geqslant 2$, then we have
		\[
		\sum\limits_{\substack{d \geqslant Y \\ p \mid d \Rightarrow p^\alpha \mid d}} \dfrac{1}{d} \leqslant \dfrac{3\kappa^\alpha \log Y}{Y^{1-\frac{1}{\alpha}}},
		\]
		and \[
		\sum\limits_{\substack{d \geqslant Y \\ p \mid d \Rightarrow p^\alpha \mid d}} \dfrac{1}{\varphi(d)} \leqslant \dfrac{13\kappa^ \alpha \log Y \log \log Y}{Y^{1-\frac{1}{\alpha}}},
		\]
		where
		\begin{equation}\label{kappa}
			\kappa = \exp \left( \sum\limits_p \frac{1}{p \log p} \right) \leqslant 5.5.
		\end{equation}
	\end{lem}
\begin{proof}
	For arbitrary $1/\alpha < \varepsilon < 1$ we have
	\begin{multline}
		\sum\limits_{\substack{d \geqslant Y \\ p \mid d \Rightarrow p^\alpha \mid d}} \dfrac{1}{d} \leqslant \dfrac{1}{Y^{1-\varepsilon}} \prod \limits_p \left(1 + \dfrac{1}{p^{\alpha \varepsilon}} + \dfrac{1}{p^{(\alpha +1)\varepsilon}} + \ldots \right) = \dfrac{1}{Y^{1-\varepsilon}} \prod \limits_p \left(1 + \dfrac{1}{p^{\alpha \varepsilon}} \dfrac{1}{1 - p^{-\varepsilon}}\right)\\
		= \dfrac{\zeta(\alpha \varepsilon)}{Y^{1-\varepsilon}} \prod \limits_p \left( 1 + \dfrac{1}{p^{(1 + \alpha) \varepsilon} - p^{\alpha \varepsilon}} - \dfrac{1}{p^{2\alpha \varepsilon} - p^{(2\alpha-1)\varepsilon}}\right) 
		\\
		\leqslant \dfrac{\zeta(\alpha \varepsilon)}{Y^{1-\varepsilon}} \prod \limits_p \left( 1 + \dfrac{1}{p^{(1 + \alpha) \varepsilon} - p^{\alpha \varepsilon}} \right),
	\end{multline}
	where $\zeta(s)$ denotes the Riemann zeta function. Let $\alpha \varepsilon = 1 + \delta$, where $0 < \delta < 1$, then
	\begin{multline}
		\sum\limits_{\substack{d \geqslant Y \\ p \mid d \Rightarrow p^\alpha \mid d}} \dfrac{1}{d} \leqslant \dfrac{\zeta(1+\delta)Y^{\frac{\delta}{\alpha}}}{Y^{1-\frac{1}{\alpha}}} \prod \limits_p \left(1 + \dfrac{1}{p^\delta \left( p^{1 + \frac{1}{\alpha} + \frac{\delta}{\alpha}} - p \right)}    \right)\\
		\leqslant \dfrac{\zeta(1+\delta)Y^{\frac{\delta}{\alpha}}}{Y^{1-\frac{1}{\alpha}}} \prod \limits_p \left(1 + \dfrac{1}{ p^{1 + \frac{1}{\alpha}} - p }    \right).
	\end{multline}
	Setting $\delta = {1}/{\log Y}$ and using the inequalities $\zeta(1+\delta) \leqslant 1 + \frac{1}{\delta}$ and \mbox{$p^{ {1}/{\alpha}} - 1 \geqslant \frac{\log p}{\alpha}$}, we obtain
	\[
	\sum\limits_{\substack{d \geqslant Y \\ p \mid d \Rightarrow p^\alpha \mid d}} \dfrac{1}{d} \leqslant \dfrac{(1 + \log Y)e^{\frac{1}{\alpha}} }{Y^{1-\frac{1}{\alpha}}}\exp \left( \alpha \sum\limits_p \dfrac{1}{p \log p}  \right) \leqslant \dfrac{3\kappa^\alpha \log Y}{Y^{1-\frac{1}{\alpha}}}.
	\]
	Let us proceed to the proof of the second inequality. Since (see \cite{Ros})
	\[
	\varphi(d) > \dfrac{d}{e^\gamma \log \log d + \frac{3}{\log \log d}},
	\]
	for $d \geqslant 3$, then for $d \geqslant 20$ we have \[
	\varphi(d) \geqslant \dfrac{3d}{13 \log \log d}.
	\]
	Put \[	
	\varepsilon = \dfrac{1}{\alpha} \left( 1 + \dfrac{1}{\log Y} \right),
	\]
	then $0<\varepsilon<1/2$ and the function $f(Y) = \dfrac{\log \log Y}{Y^{1-\varepsilon}}$ decreases as $Y \geqslant 20$; hence we obtain
	\begin{multline}
		\sum\limits_{\substack{d \geqslant Y \\ p \mid d \Rightarrow p^\alpha \mid d}} \dfrac{1}{\varphi(d)} \leqslant \dfrac{13}{3} 	\sum\limits_{\substack{d \geqslant Y \\ p \mid d \Rightarrow p^\alpha \mid d}}  \dfrac{\log \log d}{d}\\ \leqslant \dfrac{13 \log \log Y}{3 Y^{1-\varepsilon}} \prod\limits_p \left(1 + \dfrac{1}{p^{\alpha \varepsilon}} + \dfrac{1}{p^{(1 + \alpha) \varepsilon}} + \ldots \right)
		\leqslant \dfrac{13 \kappa^\alpha \log Y \log \log Y}{Y^{1-\frac{1}{\alpha}}}.
	\end{multline}
	The lemma is proved.
\end{proof}
	\begin{lem}\label{L5}
		Let
		\begin{equation}\label{Sa}
			S_\alpha(Y) = \# \left\{n \leqslant Y: p \mid n \Rightarrow p^\alpha \mid n \right\},
		\end{equation}
		then uniformly for $\alpha \geqslant 2$ we have
		\[
		S_\alpha(Y) \ll \alpha \kappa^\alpha Y^{\frac{1}{\alpha}} \log Y,
		\]
		where $\kappa$ is defined in \eqref{kappa}.
	\end{lem}
	
	\begin{proof}
		We have
		\begin{equation}
			S_\alpha(Y) = \sum\limits_{\substack{\nu \geqslant 0 \\ \frac{Y}{2^{\nu+1}} \geqslant 20}} \sum \limits_{\substack{ \frac{Y}{2^{\nu+1}} < n \leqslant \frac{Y}{2^{\nu}} \\ p \mid n \Rightarrow p^\alpha \mid n }} 1 +O(1)
			\leqslant Y\sum\limits_{\substack{\nu \geqslant 0 \\ \frac{Y}{2^{\nu+1}} \geqslant 20}}\dfrac{1}{2^\nu} \sum \limits_{\substack{n > \frac{Y}{2^{\nu+1}} \\ p \mid n \Rightarrow p^\alpha \mid n }} \dfrac{1}{n} + O(1).
		\end{equation}
		Using Lemma \ref{L4}, we obtain
		\begin{equation}
			S_\alpha(Y) \ll Y \sum\limits_{\substack{\nu \geqslant 0 \\ \frac{Y}{2^{\nu+1}} \geqslant 20}} \dfrac{1}{ 2^\nu} \dfrac{\left( \log {Y}/{2^{\nu + 1}}\right)\kappa^\alpha}{\left( {Y}/{2^{\nu + 1}} \right)^{1-\frac{1}{\alpha}} }
			\ll Y^{\frac{1}{\alpha}} \log Y \kappa^{\alpha} \sum\limits_{\nu \geqslant 0} 2^{-\frac{\nu}{\alpha} } \ll \alpha \kappa^\alpha Y^{\frac{1}{\alpha}} \log Y.
		\end{equation}
		The proof of the lemma is complete.
	\end{proof}
	\begin{lem}\label{L6}
		For $m \geqslant 4$ equality
		\[
		\mathcal{D}_m = \dfrac{1}{2} \times \dfrac{1}{2 \uparrow \uparrow m} + O\left( \dfrac{1}{3^{2 \uparrow \uparrow ( m-1)}}\right)
		\]
		 holds, where the constant in the symbol $O$ does not exceed $3$.
	\end{lem}

\begin{proof}
	We have
	\[
	\mathcal{D}_m = D_m - D_{m+1} = D_m + \dfrac{\theta_1}{2^{2 \uparrow \uparrow m}},
	\]
	where $\theta_1 = \theta_1(m)$ and $-4 \leqslant \theta_1 \leqslant 0$. Further, we have
	\[
	D_m = 1 - \prod\limits_p \vartheta_p(m),
	\]
	where $\vartheta_p(m) = 1 + \sum_{\nu \geq 1} {\mathcal{F}_{m-1}(p^\nu)}/{p^\nu}$.
	Set $\alpha = 2 \uparrow \uparrow (m-1)$. Then, according to Lemma \ref{L3}, the equality $H(\nu) = m-1$ holds for $\nu < 3^{2 \uparrow \uparrow (m-2)}$ if and only if $\nu $ has the form $\nu = q\alpha$, where $q < 1.5^{2 \uparrow \uparrow (m-2)}$ is odd. Hence we get
	\[
	\mathcal{F}_{m-1}(p^\nu) = f_{m-2}(\nu) - f_{m-2}(\nu-1) = \begin{cases}
		-1, \ \nu = q\alpha, \\
		+1, \ \nu = q \alpha + 1, \\
		0, \ \nu \ne q\alpha \text{, and}\  \nu \ne q\alpha + 1.		
	\end{cases}
	\]
	Since $\mathcal{F}_m(n) \in \left\{-1, 0, + 1\right\}$, it follows that \begin{equation}\label{in1}
		1 - \dfrac{1}{p^\alpha} + \dfrac{1}{p^{\alpha + 1}} - \dfrac{1}{p^{3\alpha}} \leqslant \vartheta_p(m) \leqslant 1 - \dfrac{1}{p^\alpha} + \dfrac{1}{p^{\alpha + 1}},
	\end{equation}
	and \begin{equation}\label{in2}
		1 - \dfrac{1}{p^\alpha} \leqslant \vartheta_p(m) \leqslant 1.
	\end{equation}
	Let us estimate the product $\prod_p \vartheta_p(m)$. For this, we apply \eqref{in1} for $p=2$ and \eqref{in2} for $p \geqslant 3$. We have
	\[
	\vartheta_2(m) = 1 - \dfrac{1}{2^{\alpha+1}} + \dfrac{\theta_2}{8^\alpha}, \ \ \ -1 \leqslant \theta_2 \leqslant 0,
	\]
	and
	\[
	1 \geqslant \prod \limits_{p \geqslant 3} \vartheta_p(m) \geqslant \prod \limits_{p \geqslant 3} \left(1 - \dfrac{1}{p^\alpha}\right) = \sum\limits_{\substack{d=1 \\ p\mid d \Rightarrow p \geqslant 3}}^{+\infty} \dfrac{\mu(d)}{d^\alpha} \geqslant 1 - \dfrac{1}{3^\alpha} - \int\limits_3^{+\infty} \dfrac{dt}{t^\alpha} \geqslant 1 - \dfrac{2}{3^\alpha},
	\]
	which is equivalent to
	\[
	\prod\limits_{p \geqslant 3} \vartheta_p(m) = 1 + \dfrac{\theta_3}{3^\alpha}, \ \ \ -2 \leqslant \theta_3 \leqslant 0.
	\]
	
	Thus, for $\mathcal{D}_m$ we get
	\begin{multline}
		\mathcal{D}_m = 1 - \left(1 - \dfrac{1}{2^{\alpha+1}} + \dfrac{\theta_2}{8^\alpha} \right) \left( 1 + \dfrac{\theta_3}{3^\alpha}\right) + \dfrac{\theta_1}{2^{2^\alpha}}\\
		= \dfrac{1}{2^{\alpha+1}} - \dfrac{\theta_2}{3^\alpha} + \dfrac{\theta_2}{2 \cdot 6^\alpha} - \dfrac{\theta_3}{8^\alpha} - \dfrac{\theta_3\theta_2}{24^\alpha} + \dfrac{\theta_1}{2^{2^\alpha}}.
	\end{multline}
	Hence
	\[
	\mathcal{D}_m \leqslant \dfrac{1}{2^{\alpha+1}} + \dfrac{1}{3^\alpha} \left(2 + \dfrac{3^\alpha}{8^\alpha}\right) \leqslant \dfrac{1}{2^{\alpha+1}} + \dfrac{3}{3^\alpha}
	\]
	and\[
	\mathcal{D}_m \geqslant \dfrac{1}{2^{\alpha+1}} - \dfrac{1}{6^\alpha} - \dfrac{2}{24^\alpha} - \dfrac{4}{2^{2^\alpha}} \geqslant \dfrac{1}{2^{\alpha+1}} - \dfrac{3}{3^\alpha}.
	\]
The lemma is proved.
\end{proof}

\begin{lem}\label{L7} Let $A\geqslant 2$ be an arbitrary fixed number, $m\geqslant 1$, then we have
	\[
	\# \left\{p \leqslant x: H(p-1) \leqslant m \right\} = k_m \pi(x) + O_A\left(\dfrac{x}{(\log x)^A} \right),
	\]
	where $k_m = \sum\limits_{d=1}^{+\infty} {\mathcal{F}_m(d)}/{\varphi(d)}$, and the constant in the symbol $O$ depends only on $A.$
\end{lem}
\begin{proof}
	We have
	$$S:=\# \left\{p \leqslant x: H(p-1) \leqslant m \right\} = \sum\limits_{d \leqslant x-1} \mathcal{F}_m(d ) \pi(x; d, 1).$$
	Put \[
	R(x; q, a) = \pi(x; q, a) - \dfrac{\pi(x)}{\varphi(q)}.
	\]
	Then, using Lemma \ref{L1}, for some $Y \leqslant \sqrt{x}$ we derive
\begin{multline}
	S = \sum\limits_{d \leqslant Y} \mathcal{F}_m(d) \pi(x; d, 1) + O \left(  \sum\limits_{\substack{d > Y,\ p \mid d \Rightarrow p^2 \mid d}} \dfrac{x}{d} \right)\\
	= \pi(x)\sum\limits_{d \leqslant Y} \dfrac{\mathcal{F}_m(d)}{\varphi(d)}+ O \left(\sum\limits_{d \leqslant Y}  |R(x; d, 1)| + \dfrac{x \log Y}{\sqrt{Y}}\right).
\end{multline}
	Complementing the sum over $d$ to infinite sum, we obtain
	\begin{equation}
		S = \pi(x)\sum\limits_{d=1}^{+\infty} \dfrac{\mathcal{F}_m(d)}{\varphi(d)} + O \left(\sum\limits_{d \leqslant Y} |R(x; d, 1)|+ \dfrac{x \log Y}{\sqrt{Y}} + \dfrac{x\log Y \log\log Y}{\sqrt {Y}\log x}\right).
	\end{equation}
	Setting $Y = x^{{1}/{3}}$ and applying the Bombieri-Vinogradov theorem, we conclude the proof of the lemma.
\end{proof}
From the lemma just proved, we trivially obtain the following equality
\begin{equation}\label{dm}
	\# \left\{p \leqslant x: H(p-1) \geqslant m \right\} = d_m \pi(x) + O_A\left(\dfrac{x}{(\log x)^A} \right),
\end{equation}
where $d_1=1$ and
\begin{equation}\label{dmdm}
	d_m = 1 - k_{m-1} = -\sum_{d>1}\dfrac{\mathcal{F}_{m-1}(d)}{\varphi(d)},\ \ m\geq 2.
\end{equation}
\begin{lem}\label{L8}
	Let $k, m \geqslant 1$, then for $H_{k,m}(x)$ defined in \eqref{HHkmx}, we have
	\begin{equation}\label{ckm1}
		H_{k,m}(x) = c_{k, m} x + {R}_{k, m}(x),
	\end{equation}
	where \begin{equation}\label{ckm}
		c_{k, m} = \prod \limits_p \left( 1 + \dfrac{\mathcal{F}_k(p) + \mathcal{F}_m(p)}{p}   + \dfrac{\mathcal{F}_k(p^2) + \mathcal{F}_m(p^2)}{p^2} + \ldots \right),
	\end{equation}
	\begin{equation}\label{Rkm}
		R_{k, m}(x) \ll \min\limits_{k_0 \leqslant k, m_0 \leqslant m} \left(\alpha{\beta} \kappa^{\alpha + \beta} x^{\frac{1}{\alpha}+\frac{1}{\beta} - \frac{1}{\alpha\beta}} (\log x)^2 \right),\, \alpha = 2 \uparrow \uparrow k_0,\, \beta = 2 \uparrow \uparrow m_0,
	\end{equation}
	the constant $\kappa$ is defined in \eqref{kappa},
	and the constant in the symbol $\ll$ is absolute.
\end{lem}
\begin{proof}
We have
	\begin{multline}
		H_{k,m}(x) = \sum\limits_{n \leqslant x} f_k(n) f_m(n+1) \\= \sum\limits_{n \leqslant x} f_m(n+1) \sum\limits_{d \mid n} \mathcal{F}_k(d)
		= \sum\limits_{d\leqslant x}\mathcal{F}_k(d)\!\!\!\!\! \sum\limits_{\substack{n \leqslant x \\ n \equiv 0 \!\!\pmod d}} \!\!\!\!\! f_m(n+1).
	\end{multline}
	Let us set some $Y \leqslant x$, then we obtain
	\[
	H_{k,m}(x) = \sum\limits_{d \leqslant Y} \mathcal{F}_k(d)  \!\!\!\!\!\sum\limits_{\substack{n \leqslant x \\ n \equiv 0 \!\!\pmod d}} \!\!\!\!\! f_m(n+1) + \mathcal{R}_k^{(1)}(x; Y),
	\]
	where
	\[
	\mathcal{R}_k^{(1)}(x; Y) \ll \!\!\!\! \sum\limits_{\substack{d > Y \\ p \mid d \Rightarrow p^{2 \uparrow \uparrow k} \mid d}}\!\!\!\! \dfrac{x}{d} \leqslant \!\!\!\sum\limits_{\substack{d > Y \\ p \mid d \Rightarrow p^{\alpha} \mid d}} \!\!\!
	\dfrac{x}{d} \ll \dfrac{x (\log Y) \kappa^\alpha}{Y^{1-\frac{1}{\alpha}}},
	\]
	and $\alpha = 2 \uparrow \uparrow k_0$, $1 \leqslant k_0 \leqslant k$. Further, we have
	\begin{multline}
		H_{k,m}(x) = \sum\limits_{d \leqslant Y} \mathcal{F}_k(d) \sum\limits_{\substack{\delta \leqslant x+1 \\ (\delta,d)=1}} \mathcal{F}_m(\delta) \# \left\{n \leqslant x: d \mid n, \delta \mid n+1 \right\} + \mathcal{R}_k^{(1)}(x; Y) \\
		= x \sum\limits_{d \leqslant Y} \dfrac{\mathcal{F}_k(d)}{d} \sum\limits_{\substack{\delta \leqslant x+1 \\ (\delta, d) = 1}} \dfrac{\mathcal{F}_m(\delta)}{\delta} + \mathcal{R}_k^{(1)}(x; Y) + \mathcal{R}_{k, m}^{(2)}(x; Y),
	\end{multline}
	where $\mathcal{R}_{k, m}^{(2)}(x; Y)$ denotes the contribution of the remainder in the formula
	$$\# \left\{n \leqslant x: d \mid n, \delta \mid n+1 \right\} = \dfrac{x}{d\delta} + O(1)$$
	to the above sum.
	By Lemma \ref{L5} we have
	\begin{equation}
		\mathcal{R}_{k, m}^{(2)}(x; Y) \ll S_\alpha(Y)S_\beta(x)
		\ll \alpha \beta \kappa^{\alpha + \beta} Y^{\frac{1}{\alpha}} x^{\frac{1}{\beta}} \log Y \log x,
	\end{equation}
	where $\beta = 2 \uparrow \uparrow m_0, \ 1 \leqslant m_0 \leqslant m$.
	Since
	\[
	\left| x \sum\limits_{d \leqslant Y}  \dfrac{\mathcal{F}_k(d)}{d}   \sum\limits_{\substack{\delta > x+1,\, (\delta, d) = 1}} \dfrac{\mathcal{F}_m(\delta)}{\delta} \right| \ll x \sum\limits_{\substack{\delta > x + 1 \\ p \mid \delta \Rightarrow p^{2 \uparrow \uparrow m} \mid \delta}} \dfrac{1}{\delta}
	\]
	and
	\[
	\left| x \sum\limits_{d > Y}  \dfrac{\mathcal{F}_k(d)}{d}   \sum\limits_{\substack{\delta \geqslant 1,\, (\delta, d) = 1}} \dfrac{\mathcal{F}_m(\delta)}{\delta} \right| \ll x \sum\limits_{\substack{d > Y \\ p \mid d \Rightarrow p^{2 \uparrow \uparrow k} \mid d}} \dfrac{1}{d},
	\]
	it follows that
	\begin{multline}
		H_{k,m}(x) = x \sum\limits_{d=1}^{+\infty}  \dfrac{\mathcal{F}_k(d)}{d}  \sum\limits_{\substack{\delta =1 \\ (\delta, d) = 1}}^{+\infty} \dfrac{\mathcal{F}_m(\delta)}{\delta}\\
		+ O \left( \mathcal{R}_k^{(1)} (x; Y) + \mathcal{R}_{k, m}^{(2)} (x; Y) + \mathcal{R}_m^{(1)} (x; x+1)\right).
	\end{multline}
	Denote the contribution of the terms in the $O$ by the symbol $R_{k, m}(x)$, then
	\begin{equation}\label{contrib}
		R_{k, m}(x)\ll \dfrac{x \log Y \kappa^\alpha}{Y^{1-\frac{1}{\alpha}}} + \alpha \beta \kappa^{\alpha + \beta} Y^{\frac{1}{\alpha}} x^{\frac{1}{\beta}} \log Y \log x +  \kappa^\beta x^{\frac{1}{\beta}}\log x.
	\end{equation}
	Let $Y = x^\delta$, $0 < \delta < 1$, then \eqref{contrib} implies\[
	R_{k, m}(x) \ll \alpha \beta \kappa^{\alpha + \beta} (\log x)^2 \left(x^{1 + \frac{\delta}{\alpha} - \delta} + x^{\frac{\delta}{\alpha} + \frac{1}{\beta}}\right).
	\]
	Assuming $\delta = 1 - {1}/{\beta}$, we conclude the proof of \eqref{ckm1}.
	
	It remains to check equality \eqref{ckm}. We have
	\begin{multline}
		\sum\limits_{d=1}^{+\infty} \dfrac{\mathcal{F}_k(d)}{d} \sum\limits_{\substack{\delta =1 \\ (\delta, d) = 1}}^{+\infty}  \dfrac{\mathcal{F}_m(\delta)}{\delta}\\
		= \prod\limits_p \left(1 + \dfrac{\mathcal{F}_m(p)}{p} + \ldots \right) \sum\limits_{d=1}^{+\infty} \dfrac{\mathcal{F}_k(d)}{d}\prod\limits_{p \mid d} \left(1 + \dfrac{\mathcal{F}_m(p)}{p} + \ldots\right)^{-1} \\
		= \prod \limits_p \left(1 + \left( \dfrac{\mathcal{F}_k(p)}{p} + \ldots\right)\left( 1 + \dfrac{\mathcal{F}_m(p)}{p} + \ldots\right)^{-1} \right) \\
		\times\prod \limits_p \left(1 + \dfrac{\mathcal{F}_m(p)}{p} + \ldots\right)= \prod \limits_p  \left(1 + \dfrac{\mathcal{F}_k(p) + \mathcal{F}_m(p)}{p} + \ldots \right).
	\end{multline}
The lemma is proved. 
\end{proof}
Let $c_{k,m} = 0,$ if at least one of the numbers $k$ or $m$ is equal to zero. With this definition, equality \eqref{ckm1} holds for $k, m\geqslant 0.$
\begin{lem}\label{L9} Put
	\begin{equation}\label{DDkmx}
		\mathcal{D}_{k,m}(x) = \#\{n\leqslant x: H(n) = k, H(n+1) = m\},
	\end{equation}
	then for $k, m \geqslant 1$ and $x\geqslant 1$ we have
	\begin{equation}\label{Dkm}
		\mathcal{D}_{k,m}(x) = \mathcal{D}_{k, m} x + O \left(x^{{5}/{8}} (\log x)^2 \right),
	\end{equation}
	where $$\mathcal{D}_{k, m} = c_{k, m} - c_{k-1, m}-c_{k, m-1} + c_{k-1, m-1}, $$ and the constant in the symbol $O$ does not depend on $k$ and $m$.
\end{lem}
\begin{proof}
	First, let $k=m=1$. Then, according to Theorem 2 in \cite{Reus14}, we have
	\begin{multline}
		\mathcal{D}_{1,1}(x) = \sum\limits_{n \leqslant x} \mu^2(n)\mu^2(n+1)\\
		=\prod\limits_p \left(1 - \dfrac{2}{p^2}\right) x + O \left(x^{\Delta + \varepsilon}\right),\ \  \text{where}\ \  \Delta = \frac{26 + \sqrt{433}}{81} = 0.5778\ldots
	\end{multline}
	Hence for $R_{k,m}(x)$ defined in \eqref{Rkm} in the case $k=m=1$ we have
	\begin{equation}\label{R11}
		R_{1,1}(x)\ll x^{\Delta+\varepsilon}.
	\end{equation}
Further, since
	\[
	\prod \limits_p \left(1 - \dfrac{2}{p^2}\right) = \prod \limits_p \left(1 + \dfrac{\mathcal{F}_1(p) + \mathcal{F}_1(p)}{p} + \dfrac{\mathcal{F}_1(p^2) + \mathcal{F}_1(p^2)}{p^2} + \ldots \right) = c_{1, 1} = \mathcal{D}_{1, 1},
	\]
	then equality \eqref{Dkm} is proved in the case $k=m=1$.
	
	Let $k=1$ and $m \geqslant 2$, then
	\begin{multline}
		\mathcal{D}_{1,m}(x)
		= H_{1,m}(x)-H_{1,m-1}(x)+O(1)\\
		=(c_{1, m} - c_{1, m-1})x + {R}_{1, m}(x) - {R}_{1, m-1}(x) + O(1).
	\end{multline}
	Further, we have
	\[
	c_{1, m} - c_{1, m-1} = c_{1, m} - c_{1, m-1} - c_{0, m} + c_{0, m-1} = \mathcal{D}_{1, m}.
	\]
	By Lemma \ref{L8}, for $m \geqslant 3$ we have
	\[
	|{R}_{1, m}(x)|, |{R}_{1, m-1}(x)| \ll x^{{5}/{8}}(\log x)^2.
	\]
	Hence, in view of \eqref{R11}, we obtain \eqref{Dkm} in the case $k=1$, $m\geqslant 2$. The case $k \geqslant 2$, $m = 1$ is treated similarly.
	
	Let, finally, $k, m \geqslant 2$. Then
	\begin{multline}
		\mathcal{D}_{k,m}(x) =H_{k,m}(x)-H_{k-1,m}(x)-H_{k,m-1}(x)+H_{k-1,m-1}(x)\\
		=x(c_{k, m} - c_{k, m-1} - c_{k-1, m} + c_{k-1, m-1} ) \\
		+ {R}_{k, m}(x) - {R}_{k, m-1}(x) - {R}_{k-1, m}(x) +  {R}_{k-1, m-1}(x).
	\end{multline}
	In view of estimates \eqref{Rkm} and \eqref{R11} we conclude that equality \eqref{Dkm} is proved in the case $k, m\geqslant 2$. The lemma is proved.
\end{proof}
\begin{lem}\label{L10}
	We have

		$$\mathcal{D}_{k,1} = \dfrac{4}{9} \dfrac{1}{2 \uparrow \uparrow k} + O\left(\dfrac{1}{3^{2 \uparrow \uparrow (k-1)}} \right)\ \ \   \mbox{for $k\geq2$},$$
		
		$$\mathcal{D}_{1,m} =\dfrac{4}{9} \dfrac{1}{2 \uparrow \uparrow m} + O\left(\dfrac{1}{3^{2 \uparrow \uparrow (m-1)}} \right)\ \  \mbox{for $m\geq2$},$$
		
		$$\mathcal{D}_{k, 2} = \dfrac{25}{486} \dfrac{1}{2 \uparrow \uparrow k} + O\left(\dfrac{1}{3^{2 \uparrow \uparrow (k-1)}} \right)\ \  \mbox{for $k\geq3$},$$
		
		$$\mathcal{D}_{2, m} = \dfrac{25}{486} \dfrac{1}{2 \uparrow \uparrow m} + O\left(\dfrac{1}{3^{2 \uparrow \uparrow (m-1)}} \right)\ \  \mbox{for $m\geq3$},$$
		and
		\begin{equation}\mathcal{D}_{k, m} = \dfrac{1}{3} \dfrac{1}{2 \uparrow \uparrow k} \left(\dfrac{1}{3^{2 \uparrow \uparrow (m-1)}}-\dfrac{1}{3^{2 \uparrow \uparrow m}}\right)\\ + O\left(\dfrac{1}{2^{2 \uparrow \uparrow k}}\dfrac{1}{3^{2 \uparrow \uparrow (m-1)}} + \dfrac{1}{2^{2 \uparrow \uparrow (m-1)}}\dfrac{1}{3^{2 \uparrow \uparrow (k-1)}} \right)
		\end{equation}
	for $k, m\geq3$.
	\end{lem}
\begin{proof} For $c_{k,m}$ defined in \eqref{ckm}, we have
	\begin{multline}\label{ckm11}c_{k, m} = \sum_{d=1}^{+\infty}\sum\limits_{\substack{\delta=1 \\ (d, \delta) = 1}}^{+\infty} \dfrac{\mathcal{F}_k(d)\mathcal{F}_m(\delta)}{d \delta} = 1 + \sum\limits_{d>1} \dfrac{\mathcal{F}_k(d)}{d} + \sum\limits_{\delta>1} \dfrac{\mathcal{F}_m(\delta)}{\delta} \\
		+\sum\limits_{\substack{2^\alpha \leqslant d < 4^\alpha \\ 2^\beta \leqslant \delta <4^\beta \\ (d, \delta)=1}} \dfrac{\mathcal{F}_k(d)\mathcal{F}_m(\delta)}{d \delta} + \sum\limits_{\substack{d \geqslant 4^\alpha \\ \delta \geqslant 4^\beta \\ (d, \delta) = 1}} \dfrac{\mathcal{F}_k(d)\mathcal{F}_m(\delta)}{d \delta}\\
		= 1 - {D}_{k+1} - {D}_{m+1} + W_{k, m} + R_{k, m},
	\end{multline}
	where $\alpha = 2 \uparrow \uparrow k$ and $\beta = 2 \uparrow \uparrow m$.
	
	Let us calculate the sum $W_{k, m}$. First, let $k=1$ and $m \geqslant 2$, then $\alpha=2$. Lemma \ref{L1} implies $\mathcal{F}_1(d) = \mu(\Delta) \mathbb{1}_{d = \Delta^2}$, hence
	\begin{multline}
		W_{1, m} = \sum\limits_{4 \leqslant d < 16} \dfrac{\mathcal{F}_1(d)}{d} \sum\limits_{\substack{2^\beta \leqslant \delta < 4^\beta \\ (\delta, d) = 1}} \dfrac{\mathcal{F}_m(\delta)}{\delta} 
		= \sum\limits_{2 \leqslant \Delta < 4} \dfrac{\mu(\Delta)}{\Delta^2} \sum\limits_{\substack{2^\beta \leqslant \delta < 4^\beta \\ (\delta, \Delta) = 1}} \dfrac{\mathcal{F}_m(\delta)}{\delta} \\
		= - \dfrac{1}{4} \sum\limits_{\substack{2^\beta \leqslant \delta < 4^\beta \\ (\delta, 2) = 1}} \dfrac{\mathcal{F}_m(\delta)}{\delta} - \dfrac{1}{9} \sum\limits_{\substack{2^\beta \leqslant \delta < 4^\beta \\ (\delta, 3) = 1}} \dfrac{\mathcal{F}_m(\delta)}{\delta}
		= -\dfrac{1}{4} W_m^{(1)} - \dfrac{1}{9} W_m^{(2)},\ \ \ \text{say.}
	\end{multline}
Assume that $\delta$ gives a nonzero contribution to the sum $W_m^{(1)}$, then for every prime $p$ dividing $\delta$, we have $\beta\leqslant \nu_p(\delta)<2\beta $ and $p<5.$ Hence, since $\delta$ is odd, it follows that $\delta = 3^\nu,$ where $\beta\leqslant \nu<2\beta.$ Similarly, if $\delta$ gives a non-zero contribution to the sum $W_m^{(2)}$, then $\delta$ has the form $\delta = 2^\nu,$ where $ \beta\leqslant \nu<2\beta.$
	
	Further, from Lemma \ref{L1}, we obtain
	$$\mathcal{F}_m(2^\nu) = \mathcal{F}_m(3^\nu) = f_{m-1}(\nu) - f_{m-1}(\nu-1).$$
	
		Let
		$m \geqslant 3$, then $ \beta \leqslant \nu < 2 \beta \leqslant 3^{2 \uparrow \uparrow (m-1)}$. Therefore, by Lemma \ref{L3}, the equality $H(\nu) = m$ holds exactly when $\nu = \beta$, hence
		\begin{equation}\label{fmnu1}
			f_{m-1}(\nu) - f_{m-1}(\nu-1) =  \begin{cases}
				-1, \nu = \beta, \\
				+1, \nu = \beta + 1, \\
				0, \ \nu \not \in \{\beta, \beta + 1\}.
			\end{cases}
		\end{equation}
		The equality above is easily verified for $m=2$. So we get
		\[
		- \dfrac{1}{4} W_{m}^{(1)} = - \dfrac{1}{4} \sum\limits_{\nu \in \{\beta, \beta+1\}} \dfrac{\mathcal{F}_m(3^\nu)}{3^\nu} = - \dfrac{1}{4} \left( \dfrac{-1}{3^\beta} + \dfrac{1}{3^{\beta + 1}} \right) = \dfrac{1}{6} \dfrac{1}{3^\beta},
		\]
		and
		\[
		- \dfrac{1}{9} W_{m}^{(2)} = - \dfrac{1}{9} \sum\limits_{\nu \in \{\beta, \beta+1\}} \dfrac{\mathcal{F}_m(2^\nu)}{2^\nu} = - \dfrac{1}{9} \left( \dfrac{-1}{2^\beta} + \dfrac{1}{2^{\beta + 1}} \right) = \dfrac{1}{18} \dfrac{1}{2^\beta}.
		\]
		Hence
		\[
		W_{1, m} = \dfrac{1}{18} \dfrac{1}{2^\beta} +  \dfrac{1}{6} \dfrac{1}{3^\beta}.
		\]
		Proceeding as above, we can show that for $k\geqslant 2$, $\alpha\leqslant \nu<2\alpha$
		\begin{equation}\label{fmnu2}
			f_{k-1}(\nu) - f_{k-1}(\nu-1) =  \begin{cases}
				-1, \nu = \alpha, \\
				+1, \nu = \alpha + 1, \\
				0, \ \nu \not \in \{\alpha, \alpha + 1\}.
			\end{cases}
		\end{equation}
		Hence for $k\geqslant 2$ and $m=1$ we get
		\begin{multline}\label{Wk1}
			W_{k, 1} =  \sum\limits_{2^\alpha \leqslant d < 4^\alpha} \dfrac{\mathcal{F}_k(d)}{d} \sum\limits_{\substack{4 \leqslant \delta < 16 \\ (\delta, d) = 1}} \dfrac{ \mathcal{F}_1(\delta)}{  \delta} = \sum\limits_{2^\alpha \leqslant d < 4^\alpha} \dfrac{\mathcal{F}_k(d)}{d} \sum\limits_{\substack{2 \leqslant \Delta < 4 \\ (\Delta, d) = 1}} \dfrac{ \mu(\Delta)}{  \Delta^2}\\
			= -\dfrac{1}{4} W_k^{(1)} - \dfrac{1}{9} W_k^{(2)} = \dfrac{1}{18} \dfrac{1}{2^\alpha} + \dfrac{1}{6} \dfrac{1}{3^\alpha}.
		\end{multline}
		Let, finally, $k, m\geqslant 2$, then using \eqref{fmnu1} and \eqref{fmnu2}, we obtain
	\begin{multline}\label{Wkm}
		W_{k, m} = \sum\limits_{2^\alpha \leqslant d < 4^\alpha} \dfrac{\mathcal{F}_k(d)}{d} \sum\limits_{\substack{2^\beta \leqslant \delta < 4^\beta \\ (\delta, d) = 1}} \dfrac{ \mathcal{F}_m(\delta)}{  \delta}\\
		= \left( \dfrac{\mathcal{F}_k(2^\alpha)}{2^\alpha} + \dfrac{\mathcal{F}_k(2^{\alpha + 1})}{2^{\alpha + 1}} \right)  \times  \left( \dfrac{\mathcal{F}_m(3^\beta)}{3^\beta} + \dfrac{\mathcal{F}_m(3^{\beta + 1})}{3^{\beta + 1}} \right)\\
		+  \left( \dfrac{\mathcal{F}_k(3^\alpha)}{3^\alpha} + \dfrac{\mathcal{F}_k(3^{\alpha + 1})}{3^{\alpha + 1}} \right)  \times  \left( \dfrac{\mathcal{F}_m(2^\beta)}{2^\beta} + \dfrac{\mathcal{F}_m(2^{\beta + 1})}{2^{\beta + 1}} \right)\\
		= \dfrac{1}{3}\left(\dfrac{1}{2^\alpha } \dfrac{1}{3^\beta } + \dfrac{1}{2^\beta } \dfrac{1}{3^\alpha }\right).
	\end{multline}
	Let us estimate $R_{k,m}$. We have
	\[
	|R_{k, m}| \leqslant \sum\limits_{d \geqslant 4^\alpha} \sum\limits_{\substack{\delta\geqslant 4^\beta  \\ (d, \delta) = 1}} \left|\dfrac{\mathcal{F}_k(d) \mathcal{F}_m(\delta)}{d \delta}\right| \leqslant V_kV_m,
	\]
	where $V_k = \sum\limits_{d \geqslant 4^\alpha} \dfrac{|\mathcal{F}_k(d)|}{d}$, and $V_m = \sum\limits_{d \geqslant 4^\beta} \dfrac{|\mathcal{F}_m(\delta)|}{\delta}$.
	
	Let us estimate the sum $V_k$ ($V_m$ can be estimated in a similar way). For $k \in\{1, 2\}$ we obviously have $V_k \ll 1$. Let $k \geqslant 3$, then
	\begin{equation}\label{Vk}
		V_k = \sum\limits_{4^\alpha \leqslant d < 32^\alpha} \dfrac{|\mathcal{F}_k(d)|}{d} + O \left(\sum\limits_{\substack{d \geqslant 32^\alpha \\ p \mid d \Rightarrow p^\alpha \mid d}} \dfrac{1}{d}\right).
	\end{equation}
	Let us consider the first sum. Since $4^\alpha \leqslant d < 32^\alpha$, then for every prime $p$, for which $p^\nu || d$, we have $2 \leqslant p \leqslant 31$ and $\alpha\leqslant \nu<5\alpha$. Since
	$\alpha \leqslant \nu < 5 \alpha \leqslant 3^{2 \uparrow \uparrow (k-1)}$, then, by Lemma \ref{L3}, equality $H(\nu) = k$ holds if and only if $\nu \in \{\alpha, 3 \alpha\}$. So we get
	\[
	\mathcal{F}_k(p^\nu) = f_{k-1}(\nu) - f_{k-1}(\nu-1) = \begin{cases}
		-1,  \nu \in \{\alpha, 3 \alpha\}, \\
		+1, \nu \in \{\alpha+1, 3 \alpha+1\}, \\
		0, \nu \not \in \{\alpha, \alpha+1, 3 \alpha, 3 \alpha +1 \}.
	\end{cases}
	\]
	Thus, each $d$ that gives a nonzero contribution to the first sum in \eqref{Vk} has the form
	$$d = q_1^{\nu_1} \ldots q_{11}^{\nu_{11}},$$
			where $q_1 = 2$, $q_2 = 3$, $\ldots$, $q_{11} = 31$, $\nu_i \in \{0, \alpha, \alpha + 1, 3 \alpha, 3\alpha+1\}$, and at least one of $\nu_i$ is non-zero.
			Since there are at most $ 5^{11}$ such numbers, then
			\[
			\sum\limits_{4^\alpha \leqslant d < 32^\alpha} \dfrac{|\mathcal{F}_k(d)|}{d} \leqslant  \dfrac{5^{11}}{4^\alpha}.
			\]
			For the second sum in \eqref{Vk} we have
			\[
			\sum\limits_{\substack{d \geqslant 32^\alpha \\ p \mid d \Rightarrow p^\alpha \mid d}} \dfrac{1}{d} \ll \dfrac{\alpha \kappa^\alpha}{32^\alpha} \ll \dfrac{1}{4^\alpha},
			\]
			hence $V_k\ll 4^{-\alpha}.$ Similarly, we get $V_m \ll {4^{-\beta}}$, thus
			\begin{equation}\label{Rkm1}
				|R_{k, m}| \ll {4^{-\alpha - \beta}}.
			\end{equation}
			
			Now we compute the asymptotics for $\mathcal{D}_{k, m}$. Let $k\geqslant 3$ and $m=1$, then by \eqref{ckm11} we have
			\begin{multline}\label{k1}
				\mathcal{D}_{k, 1} = c_{k,1}-c_{k-1,1}=D_k-D_{k+1}+W_{k,1}-W_{k-1,1}+R_{k,1}-R_{k-1,1}\\
				= \dfrac{1}{2} \dfrac{1}{2 \uparrow \uparrow k}- \dfrac{1}{18} \dfrac{1}{2 \uparrow \uparrow k} + O \left( \dfrac{1}{3^{2 \uparrow \uparrow (k-1)}}\right) +  O \left( \dfrac{1}{4^{2 \uparrow \uparrow (k-1)}}\right) \\
				=\dfrac{4}{9} \dfrac{1}{2 \uparrow \uparrow k} +  O \left( \dfrac{1}{3^{2 \uparrow \uparrow (k-1)}}\right).
			\end{multline}
			Since $\mathcal{D}_{2,1} = O(1),$ then \eqref{k1} also holds for $k=2.$
			Similarly, for $m\geqslant 2$ we have
			\[
			\mathcal{D}_{1, m} = \dfrac{4}{9} \dfrac{1}{2 \uparrow \uparrow m} + O \left( \dfrac{1}{3^{2 \uparrow \uparrow (m-1)}}\right).
			\]
			
			 Now let $k\geqslant 3$ and $m=2$. Then from \eqref{ckm11}, \eqref{Wk1}, \eqref{Wkm} and \eqref{Rkm1} we deduce
			\begin{multline}
				\mathcal{D}_{k,2} = W_{k, 2} - W_{k-1, 2} + W_{k-1, 1} - W_{k, 1} + R_{k, 2} - R_{k-1, 2} + R_{k-1, 1} - R_{k, 1}\\
				=\dfrac{1}{3}\left( \dfrac{1}{2^{2 \uparrow \uparrow k}}\dfrac{1}{3^{2 \uparrow \uparrow 2}}+ \dfrac{1}{2^{2 \uparrow \uparrow 2}}\dfrac{1}{3^{2 \uparrow \uparrow k}}\right) - \dfrac{1}{3}\left( \dfrac{1}{2^{2 \uparrow \uparrow (k-1)}}\dfrac{1}{3^{2 \uparrow \uparrow 2}}+ \dfrac{1}{2^{2 \uparrow \uparrow 2}}\dfrac{1}{3^{2 \uparrow \uparrow (k-1)}}\right)\\
				+ \dfrac{1}{18}\dfrac{1}{2^{2 \uparrow \uparrow (k-1)}} + \dfrac{1}{6} \dfrac{1}{3^{2 \uparrow \uparrow (k-1)}} - \dfrac{1}{18}\dfrac{1}{2^{2 \uparrow \uparrow k}} - \dfrac{1}{6} \dfrac{1}{3^{2 \uparrow \uparrow k}}+O\left(\dfrac{1}{4^{2\up\up(k-1)}} \right)  \\
				=\left(\dfrac{1}{18} - \dfrac{1}{3^5} \right)   \dfrac{1}{2 \uparrow \uparrow k} + O \left(\dfrac{1}{3^{2 \uparrow \uparrow (k-1)}} \right)\\
				= \dfrac{25}{486}  \dfrac{1}{2 \uparrow \uparrow k} + O \left(\dfrac{1}{3^{2 \uparrow \uparrow (k-1)}} \right).
			\end{multline}
			Similarly, for $m\geqslant 3$, we get
			$$
			\mathcal{D}_{2, m} = \dfrac{25}{486}  \dfrac{1}{2 \uparrow \uparrow m} + O \left(\dfrac{1}{3^{2 \uparrow \uparrow (m-1)}} \right).
			$$
			The remaining case is $k, m \geqslant 3$. Equality \eqref{ckm11} implies
			\begin{multline}
				\mathcal{D}_{k, m} = W_{k, m} - W_{k-1, m} - W_{k, m-1} + W_{k-1, m-1}\\
				+R_{k, m} - R_{k-1, m} - R_{k, m-1} + R_{k-1, m-1}.
			\end{multline}
			Further, from \eqref{Wkm} we obtain
			\begin{multline}
				W_{k, m} - W_{k-1, m} - W_{k, m-1} + W_{k-1, m-1} \\
				=\dfrac{1}{3} \left( \left( \dfrac{1}{2^{2 \uparrow \uparrow k}} - \dfrac{1}{2^{2 \uparrow \uparrow (k-1)}}\right) \dfrac{1}{3^{2 \uparrow \uparrow m}} + \left( \dfrac{1}{3^{2 \uparrow \uparrow k}} - \dfrac{1}{3^{2 \uparrow \uparrow (k-1)}}\right) \dfrac{1}{2^{2 \uparrow \uparrow m}}\right)\\
				- \dfrac{1}{3} \left( \left( \dfrac{1}{2^{2 \uparrow \uparrow k}} - \dfrac{1}{2^{2 \uparrow \uparrow (k-1)}}\right) \dfrac{1}{3^{2 \uparrow \uparrow (m-1)}} + \left( \dfrac{1}{3^{2 \uparrow \uparrow k}} - \dfrac{1}{3^{2 \uparrow \uparrow (k-1)}}\right) \dfrac{1}{2^{2 \uparrow \uparrow (m-1)}}\right)\\
				= \dfrac{1}{3}  \dfrac{1}{2 \uparrow \uparrow k}\left( \dfrac{1}{3^{2 \uparrow \uparrow (m-1)}} -  \dfrac{1}{3^{2 \uparrow \uparrow m}}\right) + O\left( \dfrac{1}{2^{2 \uparrow \uparrow k}}  \dfrac{1}{3^{2 \uparrow \uparrow (m-1)}}+ \dfrac{1}{2^{2 \uparrow \uparrow (m-1)}}  \dfrac{1}{3^{2 \uparrow \uparrow (k-1)}} \right).
			\end{multline}
			Since
			$$R_{k, m} - R_{k-1, m} - R_{k, m-1} + R_{k-1, m-1}\ll \dfrac{1}{4^{2\up\up(k-1)+2\up\up(m-1)}},$$
			it follows that the lemma is proved in the case $k, m\geqslant 3$. The claim follows.

			
			
			\end{proof}
		\section{Proof of Theorem \ref{T1}}
	
	By Lemma \ref{L7} we have
	\begin{multline}
		\mathcal{M}(x) = \sum\limits_{m \leqslant \log^*x} m \left(\# \{p \leqslant x: H(p-1) \geqslant m \} -  \# \{p \leqslant x: H(p-1) \geqslant m+1 \}\right)\\
		=\pi(x)  \sum\limits_{m \leqslant \log^*x} m (d_m-d_{m+1}) +  O_A \left(\dfrac{x}{(\log x)^A} \right).
	\end{multline}
	Hence, applying partial summation and complementing the sum over $m$ to infinite sum, we get
	\begin{multline}\label{M}
		\mathcal{M}(x) = \pi(x)  \sum\limits_{m \leqslant \log^*x}  d_m + O_A \left( \dfrac{x \log^*x}{\log x} d_{\log^*x+1} + \dfrac{x}{(\log x)^A}\right) \\
		= \pi(x) \sum\limits_{m=1}^{+\infty} d_m + O_A \left( \dfrac{x}{\log x} \sum\limits_{m \geqslant \log^*x+1} d_m + \dfrac{x \log^* x}{\log x}   d_{\log^*x+1} + \dfrac{x}{(\log x)^A} \right).
	\end{multline}
	
	Now we prove the estimate $d_m \ll \frac{1}{2 \uparrow \uparrow m}$. This is obvious for $1\leq m\leq 3$, so we assume that $m\geq 4$. Then by \eqref{dmdm} we have
	$d_m = d_m^{(1)} + d_m^{(2)},$
	where
	\[
	d_m^{(1)} = - \sum\limits_{2^\alpha \leqslant d < 32^\alpha} \dfrac{\mathcal{F}_{m-1}(d)}{\varphi(d)}
	\]
	and
	\[
	d_m^{(2)} = - \sum\limits_{ d\geqslant 32^\alpha  } \dfrac{\mathcal{F}_{m-1}(d)}{\varphi(d)}.
	\]
	As in the proof of Lemma \ref{L10}, a number $d$ gives a nonzero contribution to the
	sum $d^{(1)}_m$ only if it has the form
	$$d = q_1^{\nu_1}q_2^{\nu_2}\ldots q_{11}^{\nu_{11}},$$
			where $q_i$ is $i$-th prime number, $\nu_i\in\{0, \alpha, \alpha+1, 3\alpha, 3\alpha+1\}$, and at least one of $\nu_i$ is nonzero, $\alpha = 2\up\up(m-1)$, and $1\leqslant i\leqslant 11.$
			Since there are at most $5^{11} \leq 2^{26}$ such numbers, then 
	\[
	|d_m^{(1)}| \leqslant \dfrac{2^{26}}{\varphi(\delta)},
	\]
	where 
	$$\varphi(\delta) = \min\{\varphi(d): 2^\alpha \leqslant d < 32^\alpha,  \mathcal{F}_{m-1 }(d)\neq 0\}.$$
	If $p^\nu || \delta$ then $\nu \geqslant \alpha$ and $\varphi(\delta) \geqslant \varphi(p^\nu) \geqslant \varphi(2^\alpha) = 2^{\alpha-1} $. Hence 
	\[
	|d_m^{(1)}| \leqslant  \dfrac{2^{27}}{2^\alpha}.
	\]
	Further, from Lemma \ref{L4} for $\alpha\geqslant 16$ we obtain
	\[
	|d_m^{(2)}| \leqslant \dfrac{13\kappa^\alpha \log 32^\alpha \log \log 32^\alpha}{(32^\alpha)^{1-\frac{1}{\alpha}}} \leqslant \dfrac{2^{27}}{2^\alpha}.
	\]
	So for $m\geqslant 4$ we have
	\begin{equation}\label{dm_ineq}
		d_m \leqslant \dfrac{2^{28}}{2 \uparrow \uparrow m}.
	\end{equation}
	Substituting the above inequality into \eqref{M} and using $$2\up\up (\log^* x+1)\geqslant x,$$ we complete the proof of Theorem \ref{T1}.
	\section{Proof of Theorem \ref{T2}}
	
	Using \eqref{DDkx} we obtain
	\[
	\mathcal{H}(x) = \sum\limits_{k \leqslant \log^* x - 1} (2 \uparrow \uparrow k)\mathcal{D}_k(x) + H_0(x),
	\]
	where $H_0(x) =  (2 \uparrow \uparrow \log^* x) \mathcal{D}_{\log^* x}(x)$.
Equality \eqref{Dkx} implies
	\begin{multline}
		\mathcal{H}(x) = \sum\limits_{k \leqslant \log^* x - 1} 2 \uparrow \uparrow k \left(\mathcal{D}_k x + O \left( x^{1/2}\left( \log x\right) ^{{5}/{4}} \right) \right) + H_0(x) \\
		= \dfrac{x \log^* x}{2} + x \sum\limits_{k \leqslant \log^* x } \left( (2 \uparrow \uparrow k)\mathcal{D}_k - \dfrac{1}{2} \right) + H_1(x) + O \left( x^{1/2}\left( \log x\right) ^{{9}/{4}} \right)\\
		=  \dfrac{x \log^* x}{2} + c_1x + H_1(x) + O \left( x^{1/2}\left( \log x\right) ^{{9}/{4}}  + x \!\!\!\sum\limits_{k \geqslant \log^* x + 1} \left(\frac{2}{3} \right)^{2 \uparrow \uparrow (k-1)}\right),
	\end{multline}
	where $c_1$ and $H_1(x)$ are defined in the assumption. Since
	\[
	\sum\limits_{k \geqslant \log^* x + 1} \left(\dfrac{2}{3} \right)^{2 \uparrow \uparrow (k-1)} \ll \left(\dfrac{2}{3} \right)^{2 \uparrow \uparrow \log^* x} \leqslant \left(\dfrac{2}{3} \right)^{\frac{\log x}{\log 2}} \leqslant \dfrac{1}{\sqrt{x}},
	\]
	then \eqref{HH} is proved. Now let us prove the inequality $H_1(x) \ll x$. By Lemma \ref{L2} and Lemma \ref{L6} we have
	\[
	|H_1(x)| \leqslant (2 \uparrow \uparrow \log^* x) D_{\log^* x}(x) + (2 \uparrow \uparrow \log^*x) x\,\mathcal{D}_{\log^* x}\ll x.
	\]
	Now let $x = (2 \uparrow \uparrow N)-1$, then
	\begin{multline}
		H_1(x) = 2 \uparrow \uparrow (N-1) \left(\mathcal{D}_{N-1}x + O \left( x^{1/2}\left( \log x\right) ^{{5}/{4}} \right)  \right)\\
		- (2 \uparrow \uparrow (N-1)) \mathcal{D}_{N-1} x\ll x^{1/2} \left( \log x\right) ^{{9}/{4}}. 
	\end{multline}
	Finally, let $x = 2 \uparrow \uparrow N$, then
	\begin{multline}
		H_1(x) = 2 \uparrow \uparrow N - (2 \uparrow \uparrow N) \mathcal{D}_N \, x 	=x - \left(\dfrac{1}{2} + O \left( \left(\dfrac{2}{3} \right)^{2 \uparrow \uparrow (N-1)} \right) \right) x\\
		= \dfrac{x}{2} + O\left(x \left(\dfrac{2}{3} \right)^{\frac{\log x}{\log 2}} \right) = \dfrac{x}{2} +  O  \left(x^{2 - \frac{\log 3}{\log 2}} \right). 
	\end{multline}	
	Theorem \ref{T2} is proved.
	\section{Proof of Theorem \ref{T3}}
	Using the inequality
	\begin{equation}\label{DD}
		\mathcal{D}_{k,m}(x) \leqslant D_k(x),\ \ \ k, m\geqslant 1, 
	\end{equation}
	 we have
	\begin{multline}
		\mathcal{L}(x) = \sum\limits_{k \leqslant \log^* x }\sum\limits_{m \leqslant \log^* (x+1) } \dfrac{2 \uparrow \uparrow k}{2 \uparrow \uparrow m} \mathcal{D}_{k,m}(x)\\
		= \sum\limits_{k \leqslant \log^* x } \sum\limits_{m=1}^{+\infty}  \dfrac{2 \uparrow \uparrow k}{2 \uparrow \uparrow m} \mathcal{D}_{k,m}(x)
		+ O \left( \sum\limits_{k \leqslant \log^* x }  \sum\limits_{m \geqslant \log^* (x+1)+1 } \dfrac{2 \uparrow \uparrow k}{2 \uparrow \uparrow m} D_k(x)\right)\\
		= \sum\limits_{k \leqslant \log^* x-1  }  \sum\limits_{m=1}^{+\infty}  \dfrac{2 \uparrow \uparrow k}{2 \uparrow \uparrow m} \mathcal{D}_{k,m}(x)
		+ L_0(x) + O(\log^* x),
	\end{multline}
	where $L_0(x)$ is given by $$L_0(x) = 2\up\up \log^* x\sum_{m=1}^{+\infty} \dfrac{\mathcal{D}_{\log^* x, m}(x)}{2\up\up m}.$$
	 Further, Lemma \ref{L9} implies
	\begin{multline}\label{L(x)}
		\mathcal{L}(x) = 
		x \!\!\!\sum\limits_{k \leqslant \log^* x-1  } \sum\limits_{m=1}^{+\infty} \dfrac{2 \uparrow \uparrow k}{2 \uparrow \uparrow m} \mathcal{D}_{k, m} + L_0(x) +  O \left( x^{{5}/{8}} (\log x)^3 \right)\\
		= x S + L_1(x) +  O \left( x^{{5}/{8}} (\log x)^3 \right),
	\end{multline}
	where
	$$S = \sum\limits_{k \leqslant \log^* x }\sum\limits_{m=1}^{+\infty} \dfrac{2 \uparrow \uparrow k}{2 \uparrow \uparrow m} \mathcal{D}_{k, m},$$
	and
	$L_1(x)$ is defined in the theorem.
	
	Let us find the asymptotic of $S$. To do this, we represent $S$ as a sum of several terms. We have $S = S_1+S_2+S_3+\mathcal{C}_4+\mathcal{C}_5,$
	where
	\[
	S_1 =  \sum\limits_{3 \leqslant k \leqslant \log^* x } \sum\limits_{m\geqslant 3} \dfrac{2 \uparrow \uparrow k}{2 \uparrow \uparrow m}  \mathcal{D}_{k, m},  
	\]
	\[
	S_2 = \dfrac{1}{4}\sum\limits_{3 \leqslant k \leqslant \log^* x } {(2 \uparrow \uparrow k)} \mathcal{D}_{k, 2},\ \   S_3 = \dfrac{1}{2}\sum\limits_{3 \leqslant k \leqslant \log^* x } ({2 \uparrow \uparrow k}) \mathcal{D}_{k, 1},
	\]
		\begin{equation}\label{C4C5}
		\mathcal{C}_4 =  4 \sum\limits_{m=1}^{+\infty} \dfrac{\mathcal{D}_{2, m}}{2 \uparrow \uparrow m}\ \ \text{and}\ \  \mathcal{C}_5 = 2 \sum\limits_{m=1}^{+\infty} \dfrac{\mathcal{D}_{1, m}}{2 \uparrow \uparrow m}.
	\end{equation}
	Set $$R'_{k,m} = \dfrac{2\up\up k}{2\up\up m} \mathcal{D}_{k,m}-\dfrac{1}{3}\times \dfrac{1}{2 \uparrow \uparrow m}\left( \dfrac{1}{3^{2 \uparrow \uparrow (m-1)}} - \dfrac{1}{3^{2 \uparrow \uparrow m}} \right),$$
	then we have
	\begin{multline}
		S_1 = \sum\limits_{3 \leqslant k \leqslant \log^* x } \sum\limits_{m\geqslant 3} \left( \dfrac{1}{3}\times\dfrac{1}{2 \uparrow \uparrow m} \left( \dfrac{1}{3^{2 \uparrow \uparrow (m-1)}} - \dfrac{1}{3^{2 \uparrow \uparrow m}} \right)+R'_{k,m} \right)\\
		= \mathcal{C}_0 (\log^*x - 2) + \sum\limits_{k\geqslant 3} \sum\limits_{m \geqslant 3} R'_{k, m} + O\left(\sum\limits_{k\geqslant \log^* x+1}\sum\limits_{m \geqslant 3} \left|R'_{k, m}\right| \right), 
	\end{multline}
	where $$\mathcal{C}_0 = \dfrac{1}{3}\sum_{m=3}^{+\infty}\dfrac{1}{2 \uparrow \uparrow m} \left( \dfrac{1}{3^{2 \uparrow \uparrow (m-1)}} - \dfrac{1}{3^{2 \uparrow \uparrow m}} \right).$$
	From Lemma \ref{L10} we obtain
	\[
	R'_{k, m} \ll \dfrac{2\up\up k}{2 \uparrow \uparrow m}\left(\dfrac{1}{2^{2 \uparrow \uparrow k}}\dfrac{1}{3^{2 \uparrow \uparrow (m-1)}} + \dfrac{1}{2^{2 \uparrow \uparrow (m-1)}}\dfrac{1}{3^{2 \uparrow \uparrow (k-1)}}\right) \ll \dfrac{1}{(2 \uparrow \uparrow m)^2}\left(\dfrac{2}{3}\right)^{2\up\up(k-1)},
	\]
	so 
	\[
	\sum\limits_{k\geqslant \log^* x+1}\sum\limits_{m \geqslant 3} \left|R'_{k, m}\right| \ll \sum\limits_{k \geqslant \log^* x + 1}\left(\dfrac{2}{3} \right)^{2 \uparrow \uparrow (k-1)}  \ll \left(\dfrac{2}{3} \right)^{2 \uparrow \uparrow \log^* x} \ll \dfrac{1}{\sqrt{x}}
	\]
	and
	\begin{equation}\label{S1}
		S_1 = \mathcal{C}_0\log^{*}x + \mathcal{C}_1 + O\left(\dfrac{1}{\sqrt{x}}\right),
	\end{equation}
	where $$\mathcal{C}_1 = \sum_{k\geqslant 3}\sum_{m\geqslant 3}R'_{k,m}-2\mathcal{C}_0.$$
	For the sum $S_2$ we have
	\begin{multline}\label{S2}
		S_2 = \dfrac{1}{4} 	\sum\limits_{3 \leqslant k \leqslant \log^* x }\left( \dfrac{25}{486} + \left( (2 \uparrow \uparrow k) \mathcal{D}_{k, 2} - \dfrac{25}{486}\right)\right)\\
		=\dfrac{25}{1944} (\log^* x - 2) + \sum_{k=3}^{+\infty}\left( (2 \uparrow \uparrow k)\mathcal{D}_{k, 2} - \dfrac{25}{486}\right)+O \left( \sum\limits_{k \geqslant \log^*x + 1} \left(\dfrac{2}{3} \right)^{2 \uparrow \uparrow (k-1)} \right)\\
		= \dfrac{25}{1944}\log^* x + \mathcal{C}_2 + O \left(\dfrac{1}{\sqrt{x}}\right),
	\end{multline}
	where 
	$$\mathcal{C}_2 = -\dfrac{25}{972}+\sum_{k=3}^{+\infty}\left( (2 \uparrow \uparrow k) \mathcal{D}_{k, 2} - \dfrac{25}{486}\right).$$
	Similarly, for the sum $S_3$ we have
	\begin{multline}\label{S3}
		S_3 = \dfrac{1}{2} 	\sum\limits_{2 \leqslant k \leqslant \log^* x }\left(\dfrac{4}{9} + \left( (2 \uparrow \uparrow k ) \mathcal{D}_{k, 1} - \dfrac{4}{9}\right)\right)\\
		=  \dfrac{2}{9} (\log^*x -1) + \sum\limits_{k = 2}^{+\infty}\left( (2 \uparrow \uparrow k ) \mathcal{D}_{k, 1} - \dfrac{4}{9} \right) + O \left( \dfrac{1}{\sqrt{x}}\right)\\
		=\dfrac{2}{9} \log^* x +\mathcal{C}_3 + O \left( \dfrac{1}{\sqrt{x}}\right),
	\end{multline}
	where $$\mathcal{C}_3 = -\dfrac{2}{9} + \sum\limits_{k = 2}^{+\infty}\left( (2 \uparrow \uparrow k) \mathcal{D}_{k, 1} - \dfrac{4}{9} \right).$$
	Putting together the equalities \eqref{C4C5}, \eqref{S1}, \eqref{S2} and \eqref{S3}, we get
	\begin{equation}\label{S}
		S = c_2\log^*x + c_3 + O \left( \dfrac{1}{\sqrt{x}} \right),
	\end{equation}
	where $c_2$ is defined in the theorem, and the constant $c_3$ is defined by
	$$c_3 = \mathcal{C}_1+  \mathcal{C}_2+ \mathcal{C}_3+ \mathcal{C}_4+ \mathcal{C}_5.$$
	Substitute \eqref{S} into \eqref{L(x)}, then we get \eqref{LL(x)}, as required.
	
	Now we prove $L_1(x)\ll x$. To do this, we note that from \eqref{DD}, Lemma \ref{L2} and Lemma \ref{L9}, it follows that
	$$\mathcal{D}_{k,m}x + O\left(x^{5/8}(\log x)^2\right)\leqslant D_k x+O\left(x^{1/2}(\log x)^{5/4}\right).$$
	Dividing this inequality by $x$ and taking the limit as $x\to +\infty,$ we get
	$$\mathcal{D}_{k,m}\leqslant D_k.$$
	Using this estimate and Lemma \ref{L2} we obtain
	$$L_1(x) \leqslant 2\up\up \log^*x\sum_{m=1}^{+\infty}\dfrac{1}{2\up\up m}\left(D_{\log^*x}(x)+xD_{\log^* x}\right)\ll x.$$
Finally, let $x = (2\up\up N)-1$, then from  Lemma \ref{L9} we get
	$$L_1(x)\ll 2\up\up(N-1)\sum_{m=1}^{+\infty}\left( \dfrac{1}{2\up\up m}\times x^{5/8}(\log x)^2\right) \ll x^{5/8}(\log x)^3.$$
	Theorem \ref{T3} is proved.
	
	\section{Numerical Values of the Constants}
	In this section we justify the numerical values of the constants $c_0$ and $c_1$ given in the
	Theorems \ref{T1} and \ref{T2}.
	This table contains the numerical values of ${d}_m$$ \,(2\leqslant m\leqslant 4)$  calculated numerically
	\[
	\begin{tabular}{|c|c|c|c|c|c|c|c|c|c|c|c|c|}
		\hline 
		$m$  &$2$ &$3$ &$4$   \\ \hline 
		${d}_m$  & $0.62604418\ldots$ & $0.08229031\ldots$ & $0.00000764\ldots $   \\ \hline
	\end{tabular}
	\] 
	Since $d_1=1$ and
	$$\sum_{m = 5}^{+\infty} d_m\leqslant 2^{28}\sum_{m\geqslant 65\,536}\dfrac{1}{2^m}\leqslant 10^{-19\,000},$$
	we obtain $c_0 = \sum_{m\geqslant 1} d_m = 1.7083\ldots$
	The following table contains the numerical values of $\mathcal{D}_k$ ($1\leqslant k\leqslant 4$) calculated numerically
   \[
	\begin{tabular}{|c|c|c|c|c|c|c|c|c|c|c|c|c|}
		\hline 
		$k$ &$1$ &$2$ &$3$ &$4$   \\ \hline 
		${\mathcal{D}}_k$ & 0.607927\ldots & $ 0.347995\ldots$ & $0.044069\ldots$ & $0.00000764\ldots $   \\ \hline
	\end{tabular}
	\] 
	Since
	$$\sum_{k= 5}^{+\infty} {\left|(2\up\up k)\mathcal{D}_k-\dfrac{1}{2}\right|}\leqslant 3\sum_{k =  65\,536}^{+\infty}\left( \dfrac{2}{3}\right)^{k} \leqslant 10^{-10\,000},$$
	we obtain
	$$c_1 = \sum_{k=1}^{+\infty}\left((2\up\up k)\mathcal{D}_k-\dfrac{1}{2}\right) = 1.813\ldots$$

\end{document}